\documentclass[a4paper]{amsart}
\usepackage[utf8]{inputenc}
\usepackage[T1]{fontenc}
\usepackage{amssymb}
\usepackage{amsmath}
\usepackage{amsthm}
\usepackage[english]{babel}
\usepackage[noadjust]{cite}
\usepackage[hang]{caption}
\usepackage{tikz}
\usepackage{tikz-cd} 
\usepackage{enumerate}
\usepackage{mathrsfs}
\usepackage{xcolor}
\usepackage{hyperref}
\usepackage{physics}
\usepackage{bbm}
\usepackage[normalem]{ulem}

\newcommand{\M}{\mathcal{M}}

\newcommand{\N}{\mathbb{N}}

\newcommand{\bn}{\mathbf{n}}

\newcommand{\MT}{\mathcal{M}_T}
\newcommand{\MC}{\mathscr{C}_T}
\newcommand{\MCe}{\mathscr{C}_T^{e}}

\newcommand{\MTe}{\mathcal{M}^e_T}

\renewcommand{\phi}{\varphi}
\newcommand{\eps}{\varepsilon}
\makeatletter
\hypersetup{
    colorlinks=true,
    linkcolor=black,
    citecolor=blue,
    urlcolor=black,
    pdftitle={Overleaf Example},
    pdfpagemode=FullScreen,
    }
\author[Sejal Babel]{Sejal Babel}
\address[Sejal Babel]{
Faculty of Mathematics and Computer Science, Jagiellonian University in Krakow, ul. \L o\-jasiewicza 6, 30-348 Krak\'ow, Poland
-- \and --
Doctoral School of Exact and Natural Sciences, Jagiellonian University, 
ul. \L o\-jasiewicza 11, 30-348 Krak\'ow, Poland
}\email{sejal.babel@doctoral.uj.edu.pl}

\author[Martha Łącka]{Martha Łącka}
\address[Martha Łącka]{
Faculty of Mathematics and Computer Science, Jagiellonian University in Krakow, ul. \L o\-jasiewicza 6, 30-348 Krak\'ow, Poland
}\email{martha.ubik@uj.edu.pl }

\theoremstyle{plain}
\newtheorem{thm}{Theorem}[section]
\newtheorem{lem}[thm]{Lemma}
\newtheorem{cor}[thm]{Corollary}

\theoremstyle{definition}
\newtheorem{defn}[thm]{Definition}

\DeclareMathOperator{\Gen}{Gen}
\makeatletter
\def\blfootnote{\gdef\@thefnmark{}\@footnotetext}

\title{On the Closedness of Ergodic Measures in a Characteristic Class}

\begin{document}
\begin{abstract}
We endow the set of all invariant measures of a topological dynamical system with a metric $\bar{\rho}$, which induces a topology stronger than the the weak$^*$-topology. Then, we study the closedness of ergodic measures within a characteristic class under this metric. Specifically, we show that if a sequence of generic points associated with ergodic measures from a fixed characteristic class converges in the Besicovitch pseudometric, then the limit point is generic for an ergodic measure in the same class. This implies that the set of ergodic measures belonging to a fixed characteristic class is closed in $\bar{\rho}$ (by a result of Babel, Can, Kwietniak, and Oprocha in \cite{Spec}).
\end{abstract}

\blfootnote{\textup{2020} \textit{Mathematics Subject Classification}: 37A05, 37Axx, 37Bxx}
\date{\today}
\maketitle

	\vspace{7mm}

	\textbf{keywords:} \emph{characteristic classes, Besicovitch pseudometric, joinings, generic points}
		\vspace{7mm}
        
\section{Introduction}

A collection of measure-preserving systems is called a \emph{characteristic class} if it is closed under the operations of taking factors and countable joinings. This notion was introduced in~\cite{KKPLT23} by Kanigowski, Kułaga-Przymus, Lemańczyk, and de la Rue to study the Veech conjecture about the M\"obius orthogonality, which implies the Sarnak conjecture. Some notable examples of characteristic classes are the collection of systems with discrete spectrum, systems with zero entropy, and rigid systems with a fixed rigidity sequence (see \cite{KKPLT23} for more examples of characteristic classes). 

Now consider a characteristic class $\mathscr{C}$ and a dynamical system $(X,T)$. Denote by $\MC(X)$, the family of all $T$-invariant measures in $\mathscr{C}$. Although the set of all $T$-invariant probability measures on $X$ forms a compact metric space with respect to the weak$^*$ topology, it is well known that usually $\MC(X)$ is not closed in weak$^*$ topology. This is because any metric generating the weak$^*$ topology fails to preserve the properties that defines a characteristic class. For instance weak$^*$-limit of a sequence of periodic measures on shift spaces can be a Bernoulli measure. Moreover, the entropy function in general is not weak$^*$-continuous. Thus, it is reasonable to consider a metric stronger than the metric inducing the weak$^*$ topology, which can preserve various desirable dynamical properties. One of such metrics is the Ornstein's $\bar d$-metric defined on the collection of shift-invariant measures on the symbolic space. 


In this paper, we consider the metric $\bar{\rho}$ which generalises Ornstein's $\bar d$-metric to compact spaces. More precisely, if $(X,T)$ is a topological dynamical system and $d$ is any compatible metric on $X$, then there is a complete metric $\bar{\rho}$ on the space of all $T$-invariant measures that introduces a topology stronger than the weak$^*$ topology. Furthermore, the topology induced by $\bar{\rho}$ does not depend on the choice of $d$. Finally, if $X=A^\infty$ is the symbolic space and $T=\sigma$ is the shift operator, then $\bar{\rho}$ is uniformly equivalent to Ornstein's $\bar{d}$ metric.

We establish that the set of all ergodic $T$-invariant probability measures belonging to a characteristic class is closed with respect to the metric $\bar{\rho}$. Specifically, we prove that if a sequence of generic points $(x_m)_{m\in\N}$ associated with a sequence of measures $(\mu_m)_{m\in\N}$, each belonging to a fixed characteristic class $\mathscr{C}$, converges to a point $x$ generic for a measure $\mu$ in the Besicovitch pseudometric, then $\mu$ belongs to~$\mathscr{C}$. Using the equivalence between the Besicovitch pseudometric on generic points and $\bar{\rho}$-metric on ergodic measures established by authors in \cite{Spec}, we conclude that set of all ergodic $T$-invariant probability measures in a fixed characteristic class is closed in $\bar{\rho}$-metric.



\section{Preliminaries}

\subsection{Notation}
Throughout this paper, a \emph{topological dynamical system} is a pair $(X,T)$, where $X$ is a compact metric space equipped with a compatible metric $d$ and $T\colon X\to X$ is a homeomorphism. For $x\in X$ and $\eps>0$ we denote by $B(x, \eps)$ the set $\{z\in X\,:\,d(z,x)<\eps\}$.
For a set $B\subset X$ and $\eps>0$, let $\partial B$ represent the boundary of $B$ and $B^{\eps}$ denote the $\eps$-hull of $B$, defined as the set $\{x\in X\,:\, \inf_{b\in B} d(x,b)<\eps\}$.
We denote by $\mathcal B(X)$ the Borel $\sigma$-algebra of $(X,d)$. For simplicity, we omit the $\sigma$-algebra from the notation, as we consider it to Borel $\sigma$-algebra throughout this work.

Given a dynamical system $(X,T)$, we denote by $\MT(X)$ the simplex of $T$-invariant Borel probability measures on $X$ and by $\MT^e(X)$ the subset of $\MT(X)$ consisting of ergodic measures. We also write $\mathcal M(X)$ for the family of all Borel probability measures on $X$. Given $x\in X$, let $\delta_x\in\mathcal M(X)$ be the Dirac measure supported at $x$. Fixing $\mu\in\MT(X)$, we obtain an invertible measure-preserving system $(X,\mu,T)$.

\subsection{Factors and Joinings}
A measure-preserving system $(Y,\nu,S)$ is called a \emph{factor} of $(X,\mu,T)$ if there exists a measurable map $\pi:X\to Y$ such that $\pi\circ T=S\circ \pi$ and $\pi^{*}(\mu)=\nu$. Let $I\subseteq \N$ be an at most countable set. Consider a family of measure-preserving systems $(X_i,\mu_i,T_i)_{i\in I}$. A \emph{joining} $\lambda$ of the sequence of measures $(\mu_i)_{i\in I}$ is a measure on the Cartesian product $\Pi_{i\in I}X_i$, whose marginals on $X_i$'s are $\mu_i$'s and 
is invariant under the map $\Pi_{i\in I}T_i$. By enumerating the index set $I$ with $\{1,2,\ldots\}$, we denote the set of all joinings of a sequence $(\mu_i)_{i\in I}$ by $J(\mu_1,\mu_2,\ldots) $ and the set of all ergodic joinings by $J^e(\mu_1,\mu_2,\ldots)$. 

The following theorem states that, in order to determine whether there exists a factor map between two measure-preserving systems, it is enough to check the existence of a special joining (see \cite[Proposition 2.7]{TDLR} for the proof).

\begin{thm}\label{thm:Factor}
Let $(X,T)$ and $(Y,S)$ be topological dynamical systems with $\mu\in \MT(X)$ and $\nu\in \mathcal{M}_{S}(Y)$. The system $(Y,\nu,S)$ is a factor of $(X,\mu,T)$ if and only if there exists a joining $\lambda\in J(\mu,\nu)$ such that
\[
\{X,\emptyset\}\otimes \mathcal{B}(Y)\subset \mathcal{B}(X)\otimes \{Y,\emptyset\} \mod \lambda.
\]
\end{thm}


\subsection{Characteristic classes}
\begin{defn}\label{Characteristic class}
A collection of measure-preserving systems $\mathscr{C}$ is called a \emph{characteristic class} if it is closed under the operations of taking factors and (countable) joinings. 
\end{defn}
Let $\mathscr{C}$ be a characteristic class. We say that a measure $\mu\in \MT(X)$ belongs to $\mathscr{C}$ if the measure preserving system $(X,\mu,T)\in \mathscr{C}$. We denote by $\MC(X)$ the collection of all $\mu\in\MT(X)$ belonging to $\mathscr{C}$ and by $\MCe(X)$ the set of ergodic measures in $\MC(X)$.

\subsection{The Besicovitch pseudometric}
Let $\bn:=(n_k)_{k=1}^{\infty}\subset\N$ be an increasing sequence. A point $x\in X$ is called \emph{quasi-generic} along $\bn$ for $\mu\in \MT(X)$ if for every $f\in \mathcal C(X)$ one has 
\[
\frac{1}{n_k}\sum_{j=0}^{n_k-1}f(T^j(x))\to \int f d\mu \quad \text{as} \quad k\to \infty.
\]
If in addition one has $n_k=k$ for every $k\in\N$, then we say that $x$ is \emph{generic} for $\mu$. We denote the set of all generic points for a measure $\mu$ by $\text{Gen}(\mu)$.   
\begin{defn}
The \emph{Besicovitch pseudometric} along $\mathbf{n}$ is defined for $x,y\in X$ by
	\begin{equation}\label{def:D_B-along}
D^{\bn}_B(x,y)=\limsup_{k\to\infty}\frac{1}{n_k}\sum_{j=0}^{n_k-1}d(T^j(x),T^j(y)).
	\end{equation}	
\end{defn}
We will use the notation $D_B$ when $n_k=k$ for every $k\in\N$.
Furthermore, in~\eqref{defn:equivalent-pseudometric-along-subseq} we define an equivalent pseudometric to the Besicovitch pseudometric along~$\bn$ introduced in \eqref{def:D_B-along}. The definition in \eqref{defn:equivalent-pseudometric-along-subseq} is analogous to the one presented in \cite[Lemma 2]{KLO2}, but here the pseudometric is defined along a~subsequence $\bn$. The equivalence can be established following the similar lines of reasoning as those presented in \cite[Lemma 2]{KLO2}.
\begin{equation}\label{defn:equivalent-pseudometric-along-subseq}
        \tilde{D}_{B}^{\bn}(x,y)=\inf\left\{ \delta > 0\,:\,\bar{d}_{\bn}\left(\left\{ j\in \N \,:\, d\big(T^j(x),T^j(y)\big)\geq \delta\right \}\right)< \delta   \right\},
\end{equation} 
where $\bar{d}_{\bn}(S)$ denotes the upper asymptotic density of the set $S\subseteq\N$ along $\bn$, that is
\[\bar{d}_{\bn}(S)=\limsup\limits_{k\to\infty}\frac{\left|S\cap\{0,1,\ldots, n_k-1\}\right|}{n_k}.\]


 \subsection{The $\bar{\rho}$-metric}\label{DBweak*}
The metric $\bar{\rho}$ on the space of all $T$-invariant probability measures on $X$ was introduced in \cite{Spec}, but the $\bar d$ metric has been studied by many authors (see e.g. \cite{Glas03}, \cite{rudolph90}). In fact, the symbolic version of our main result can be deduced from \cite{BKLR19}. The relationship between the metric $\bar{\rho}$ on the space of invariant measures and the Besicovitch pseudometric along a subsequence was investigated in \cite{Spec}.

\begin{defn}
 Let $\mu,\nu\in \MT(X)$. The metric $\bar{\rho}$ is defined as
\begin{equation*}\label{defn:rho_bar_metric}
\bar{\rho}(\mu,\nu)=\inf_{\lambda\in J(\mu,\nu)}\int d(x,y)\,\text{d}\lambda(x,y).
\end{equation*}
\end{defn}
It is known that $\bar\rho$ induces a topology stronger than the weak$^*$ topology (see \cite{Spec}).

\section{Main results}

We state without proof the following lemma, which is an extension of the result from \cite[Proposition 8.2.8]{bogachev07}, where the lemma was proven for a single measure.  
\begin{lem}\label{lem:Generating_partition} 
Let $(\nu_k)_{k=1}^{\infty}\subseteq \M(X)$. Then there exists a countable basis $\beta$ for the topology of $(X,d)$ such that for every $B\in\beta$ one has
$\nu_k(\partial B)=0 \text { for every } k\in \N$.
\end{lem}

The proof of the following result follows from the proofs of \cite[Lemma 6]{KLO2} and \cite[Theorem 15]{KLO2}. The obvious changes must be taken into account are caused by the fact that the points in Theorem \ref{thm:Besicovitch-quasi-genericity-and-ergodicity} are quasi-generic, while those in \cite{KLO2} are generic. 
    \begin{thm}\label{thm:Besicovitch-quasi-genericity-and-ergodicity}
Let $\bn=(n_k)_{k=1}^\infty\subseteq\N$ be strictly increasing sequence of positive integers. If for every $m\ge 1$ the point $x_m\in X$ is quasi-generic along  $\bn$  for $\mu_m\in\MT(X)$  
and $x\in X$ is such that
\[
\lim_{m\to\infty}
D^{\bn}_B(x,x_m)=0,
\]
then $x$ is quasi-generic along $\bn$ for some $\mu\in\MT(X)$. Furthermore, the sequence $(\mu_m)_{m\in\N}$ converges to $\mu$ with respect to the weak$^*$ topology and if $\mu_m$ is ergodic for every $m\ge 1$, then so is $\mu$.
\end{thm}

The following theorem is an extension of \cite[Lemma 3.17]{BKLR19}. The authors in \cite{BKLR19} prove the result in the case of symbolic spaces.

\begin{thm}\label{joining_thm}
 Let $(X,T)$ be a topological dynamical system. Suppose there exist $(x_m)_{m=1}^{\infty}\subseteq X$, $x\in X$, $(\mu_m)_{m=1}^{\infty}\subseteq \MT(X)$,  and $\bn=(n_k)_{k=1}^{\infty}\subset\N$ such that 
 \begin{enumerate}[(i)]
     \item for every $m\in\N$ the point $x_m$ is quasi-generic along $\bn$ for $\mu_m$,
     \item $D_B^{\bn}(x_m,x)\to 0$ as $m\to \infty$.
 \end{enumerate}
Then $x\in X$ is quasi-generic along $\bn$ for some measure $\mu\in\MT(X)$ and there exists a joining $\nu\in J(\mu_1,\mu_2,....)$ such that $(X,T,\mu)$ is a factor of $(X^{\infty},T^{\infty},\nu)$. Furthermore, if $\mu_m$ is ergodic for every $m\geq 1$, then one can require $\nu$ to be ergodic.
\end{thm}
\begin{proof}
The fact that $x$ is quasi-generic along $\bn$ for some measure $\mu\in\MT(X)$ follows from Theorem~\ref{thm:Besicovitch-quasi-genericity-and-ergodicity}.
To prove the existence of a joining with desired properties consider the point $\hat{x}:=(x,x_1,x_2,.....)\in X\times X^{\infty}$. Define a sequence $(\xi_k)_{k=1}^{\infty}\subseteq \mathcal{M}(X\times X^\infty)$ by setting 
\[
\xi_k=\frac{1}{n_k}\sum_{j=0}^{n_k-1}\delta_{(T\times T^{\infty})^{j}(\hat{x})}\text{ for every } k\geq 1.
\]
 Let $\bar{\nu}\in \mathcal{M}_{(T\times T^{\infty})} (X\times X^{\infty})$ be an accumulation point of the sequence $(\xi_k)_{k=1}^{\infty}$ with respect to the weak$^*$ topology (such a point exists as $\mathcal{M}(X\times X^\infty)$ is compact). Since $x$ is quasi-generic for $\mu$ along $\bn$ and for every $m\geq 1$ the point $x_m$ is quasi-generic for $\mu_m$ along $\bn$, we get that $\bar{\nu}\in J(\mu,\mu_1,\mu_2,...)$. By Theorem \ref{thm:Factor}, to prove that $(X,T,\mu)$ is  a factor of  $(X^\infty,T^\infty,\nu)$ where $\nu$ is the marginal of $\bar{\nu}$ on $X^{\infty}$, it is sufficient to prove that 
\begin{equation}\label{eq:sigma_algebra}
\mathcal{B}(X)\otimes \{X^\infty,\emptyset\}\subset \{X,\emptyset\}\otimes \mathcal{B}(X^\infty) \mod \bar{\nu}.
\end{equation}
Using Lemma \ref{lem:Generating_partition} we obtain a basis $\beta$ generating the Borel $\sigma$-algebra of $X$ such that for every $B\in \beta$ and every $m\in\N$ we have $\mu(\partial B)=\mu_m(\partial B)=0$. In order to prove \eqref{eq:sigma_algebra} it suffices to show that for every $B\in\beta$ there exists $A\in \mathcal B(X^\infty)$ such that $\bar{\nu}((B\times X^\infty)\Delta (X\times A))=0$. 
To this end fix $B\in\mathcal\beta$.
For every $n\geq 1$ define $A_n\in \mathcal B(X^{\infty})$ as 
\[
A_n:=(\underbrace{X\times X\times...\times X}_{n-1\text{ times }}\times B\times X\ldots).
\]
Note that $\partial(\partial B)\subseteq \partial B$ and so $\mu(\partial(\partial B))=0$. Therefore it follows from 
the Portmanteau theorem that \[\frac{1}{n_k}\sum_{i=0}^{n_k-1}\delta_{T^i(x)}(\partial B)\to\mu(\partial B)=0\text{ as }k\to\infty.\]
Fix $\eps>0$. There exists $\delta<\eps/2$ such that
\[
\limsup_{k\to \infty}\frac{1}{n_k}\left\lvert \{j\in \{0,1,\ldots, n_k-1\}\,:\, T^j(x)\in (\partial B)^{\delta}\}\right\rvert<\frac{\eps}{2}. 
\]
By our assumptions there exists $N\in \N$ such that $\tilde{D}_{B}^{\bn}(x,x_N)<\delta$, that is
\[
\bar{d}_{\bn}\left(\{ j\in \N \mid d(T^j(x),T^j(x_N))\geq \delta \}\right)< \delta.
\]
Now, let $\pi_N:X\times X^\infty\to X\times X$ denote the projection 
\[
\pi_N(y,y_1,y_2,.....y_N,....):=(y,y_N)
\]
and $\nu_N$ be the corresponding push-forward measure of $\bar{\nu}$. One has  
\begin{multline*}
\bar{\nu}\big((B\times X^\infty)\Delta(X\times A_{N})\big)=\nu_N((B\times X)\Delta (X\times B))=\\
=\lim_{k\to \infty}\frac{1}{n_k}\left\lvert\left\{j\in \{0,1,\ldots, n_k-1\}\,:\, \big(T^j(x),T^j(x_N)\big)\in (B\times B^c)\cup (B^c\times B)\right\}\right\rvert \\
\leq \limsup_{k\to \infty}\frac{1}{n_k}\left\lvert\{j\in \{0,1,\ldots, n_k-1\}\,:\, d(T^j(x),T^j(x_N))\geq \delta \}\right\rvert \\
+\limsup_{k\to \infty}\frac{1}{n_k}\lvert \{j\in \{0,1,\ldots, n_k-1\}\,:\, T^j(x)\in (\partial B)^{\delta}\}\rvert <\delta+\frac{\eps}{2}<\eps.
 \end{multline*}
Since $\eps>0$ is arbitrary and the $\sigma$-algebra under consideration forms a complete metric space with the metric $\hat{\rho}$ given by $\hat\rho(C,D)=\bar{\nu}(C\Delta D)$ (refer \cite[Theorem 1.12.6]{bogachev07} for details), there exists a set $A\in \mathcal{B}(X^\infty)$ such that $\bar{\nu}((B\times X^\infty) \Delta (X\times A))=0$. This establishes that $(X,T,\mu)$ is a factor of $(X^\infty, T^\infty, \nu)$.

For furthermore part, let $\mathcal{C}$ be a $T^\infty$-invariant sub $\sigma$-algebra of $\mathcal{B}(X^\infty)$ representing $(X,T,\mu)$ and let $\Lambda$ denote the ergodic decomposition of $\nu$. By restricting $\Lambda$ to the sub $\sigma$-algebra $\mathcal{C}$, we obtain 
\[
\mu=\nu_{\lvert_{\mathcal{C}}}=\int \lambda_{\lvert_{\mathcal{C}}} d(\Lambda(\lambda)).
\]
Since $\mu$ is ergodic (by Theorem \ref{thm:Besicovitch-quasi-genericity-and-ergodicity}) and the ergodic decomposition is unique, it follows that $\lambda_{\lvert \mathcal{C}}= \mu$ for $\Lambda$-almost every $\lambda$. In other words, $(X,T,\mu)$ is a factor of almost all ergodic components of $(X^\infty,T^\infty,\nu)$. To finish the proof note that due to the ergodicity of $\mu_m$ for $m\geq 1$, each such ergodic component belongs to $J^e(\mu_1,\mu_2,....)$.  
Therefore, we can replace $\nu$ by any of such components.
\end{proof}

 \begin{cor}
Let $\bn\subset\N$ be an increasing sequence, $\mathscr{C}$ be a characteristic class, and
\[
\Gen(\mathscr{C})=\big\{x\in \Gen(\mu) : \mu\in\MC(X)\big\}
\]
be the collection of generic points for measures in $\mathscr{C}$. Then $\Gen(\mathscr{C})$ is closed with respect to $D_B$.
\end{cor}
In \cite[Corollary 7.6]{Spec} the authors establish the following relationship between the Besicovitch pseudometric
for quasi-generic points and the metric $\bar{\rho}$ between the corresponding measures. 
\begin{lem}\label{Besicovitch_rho_bar}
    Let $(\mu_m)_{m=1}^{\infty}\subseteq \MTe(X)$ be a sequence of ergodic measures such that $\bar{\rho}(\mu_m,\mu)\to 0$ as $m\to \infty$ for some $\mu\in \MT(X)$. Then, there exist $(x_m)_{m=1}^{\infty}\subseteq X$, $x\in X$ and an increasing sequence $\bn\subset \N$ such that for every $m\geq 1$ the point $x_m$ is quasi-generic along $\bn$ for $\mu_m$ and \[
        \lim_{m\to \infty}D_B^{\bn}(x_m,x)= 0.
    \]    
\end{lem}
Using Lemma \ref{Besicovitch_rho_bar} we translate Theorem \ref{joining_thm} to the convergence of measures with respect to $\bar{\rho}$, proving that the set of ergodic measures from a fixed characteristic class forms a closed set with respect to $\bar{\rho}$.
\begin{cor}\label{Cor:Characteristic_rho_bar}
For every characteristic class $\mathscr{C}$ the set
$\MCe$ is closed with respect to the metric $\bar{\rho}$.
\end{cor}


\section*{Acknowledgments}
The authors would like to thank Mariusz Lemańczyk for pointing out that the extension of their theorem from \cite{BKLR19} could be proven, which led to the main result of this paper. We also thank Dominik Kwietniak and Piotr Oprocha for their useful suggestions and remarks. Sejal Babel was supported by NCN Sonata Bis grant no. 2019/34/E/ST1/00237. 

\bibliographystyle{plain}
\bibliography{library.bib}

\end{document}